\pdfoutput=1
\documentclass{article}

\usepackage{amsmath,amssymb,amsthm}

\usepackage[T1]{fontenc}
\usepackage{libertine}
\usepackage[libertine]{newtxmath}

\numberwithin{equation}{section}
\newtheorem{thm}{Theorem}[section]
\newtheorem{cor}[thm]{Corollary}
\newtheorem{prop}[thm]{Proposition}
\newtheorem{lem}[thm]{Lemma}

\theoremstyle{definition}

\theoremstyle{remark}
\newtheorem{rmk}[thm]{Remark}

\usepackage{microtype}

\usepackage[all]{xy}

\usepackage{spectralsequences}

\usepackage{todonotes}

\usepackage{hyperref}
\hypersetup{
  colorlinks   = true, 
  urlcolor     = blue, 
  linkcolor    = blue, 
  citecolor   = red 
}

\newcommand{\co}{\colon\thinspace}
\newcommand{\mb}[1]{\mathbb{#1}}

\newcommand{\Mod}{\mbox{-}{\rm mod}}

\DeclareMathOperator{\Ext}{Ext}
\DeclareMathOperator{\Hom}{Hom}
\DeclareMathOperator{\Map}{Map}

\DeclareMathOperator{\Sq}{Sq}

\DeclareMathOperator{\Der}{Der}
\DeclareMathOperator{\im}{Im}

\DeclareMathOperator{\TAQ}{TAQ}

\DeclareMathOperator*{\sma}{\wedge}

\title{Calculating obstruction groups for $E_\infty$ ring spectra}
\author{Tyler Lawson\thanks{The author was partially supported by NSF
    grant 1610408.}}

\begin{document}
\maketitle

\begin{abstract}
  We describe a special instance of the Goerss--Hopkins obstruction
  theory, due to Senger, for calculating the moduli of $E_\infty$ ring
  spectra with given mod-$p$ homology. In particular, for the
  $2$-primary Brown--Peterson spectrum we give a chain complex that
  calculates the first obstruction groups, locate the first potential
  genuine obstructions, and discuss how some of the obstruction
  classes can be interpreted in terms of secondary operations.
\end{abstract}

\section{Introduction}

The mod-$p$ homology of an $E_\infty$ ring spectrum $R$ comes equipped
with operations called the (Araki--Kudo--)Dyer--Lashof
operations. At the prime $2$, these take the form of additive natural
transformations
\[
  Q^s\co H_n(R) \to H_{n+s}(R)
\]
that satisfy a Cartan formula, have their own Adem relations, and
interact in a concrete way with homology operations via the Nishida
relations. These operations give extra structure which can be used to
classify existing $E_\infty$ rings and exclude certain phenomena. For
example, Hu--Kriz--May used relations between these operations in the
dual Steenrod algebra to show that the natural splitting $BP \to
MU_{(p)}$ cannot be a map of $E_\infty$ ring spectra.

Unfortunately, these primary operations alone are not enough
information to determine whether the Brown--Peterson spectrum $BP$
admits the structure of an $E_\infty$ ring, a long-standing problem in
the field. However, more is available in $H_* R$: secondary operations
that arise from relations between primary operations. For example, the
Adem relation $Q^{2n+1} Q^n = 0$ gives rise to a secondary operation
that increases degree by $3n+2$, but it is only defined on $\ker(Q^n)$
and its value is only well-defined mod the image of $Q^{2n+1}$. This
extra data provides more information that can enhance our
understanding of the objects that we have and the objects that we
cannot.

In \cite{secondary}, we used the calculation of a secondary operation
in the dual Steenrod algebra to show that the $2$-primary
Brown--Peterson spectrum $BP$ does not admit the structure of an
$E_\infty$ ring. Just as above, this secondary operation exists
because of a relation between primary operations. Unfortunately, the
relation in question is much more complicated than $Q^{2n+1} Q^n = 0$.

\begin{prop}
  \label{prop:bigrelation}
  Suppose that $A$ is an $E_\infty$-algebra over the Eilenberg--Mac
  Lane spectrum $H\mb F_2$ and $x \in \pi_2(A)$. Define the following classes:
  \begin{align*}
    y_5 &= Q^3 x\\
    y_7 &= Q^5 x\\
    y_9 &= Q^7 x\\
    y_{13} &= Q^{11} x\\
    y_8 &= Q^6 x + x^4\\
    y_{10} &= Q^8 x + x^2 Q^4 x\\
    y_{12} &= Q^{10}x + (Q^4 x)^2
  \end{align*}
  Then there is a relation
  \begin{align*}
    0 ={}& Q^{20} y_{10} + Q^{18} y_{12} + Q^{17} y_{13}  + x^4(Q^{12}
    y_{10}) + y_9^2 (Q^4 x)^2 +\\
    &y_7^2 Q^9Q^5 x + y_8^2 Q^8 Q^4 x + (Q^9 y_9) (Q^4x)^2 + (Q^{10} y_8) (Q^4 x)^2 + \\
    &y_5^2 (Q^{11} Q^7 x + Q^{10}Q^8 x + x^4 Q^6 Q^4 x)
  \end{align*}
  in $\pi_{30}(A)$.

  In particular, if $R$ is an $E_\infty$ ring then
  $A = H\mb F_2 \sma R$ is an $E_\infty$-algebra over $H\mb F_2$, and
  for $x \in H_2(R)$ this gives a relation in $H_{30} R$ between
  Dyer--Lashof operations on $x$.
\end{prop}
This relation gives rise to a secondary operation defined on a subset
of $H_2(R)$ (those elements $x$ for which the elements $y_k$ described
above all vanish) that takes values in a quotient of $H_{31}(R)$. A
calculation of this secondary operation found that its value on the
generator $\xi_1^2 \in H_2 H\mb F_2$ is unambiguously $\xi_5$ mod
decomposables, which makes it impossible for the map
$H_* BP \to H_* H\mb F_2$ to preserve secondary operations.

While this gives a rough description of the main result of
\cite{secondary}, and the size of this relation provides some mild
amusement value, most people who see this relation are inclined to ask
\emph{where on earth it came from}. There are no obvious indications
why we should focus attention on this particular secondary
operation. There are certainly many simpler ones---including ones
based solely on the Adem relations---that can be applied to elements
in $H_* BP$, and we could attempt to obtain simpler obstructions using
those. However, the author's experience has been that most
more straightforward attempts either fail to work or are difficult to
calculate. Most of the blame for this can be directed at $MU$: the
subring $H_* BP$ of the dual Steenrod algebra is the same as the image
of $H_* MU$. If our goal is to show that $H_* BP$ cannot be closed
under secondary operations, then we have to find a secondary operation
built out of a relation that holds in $H_* BP$ \emph{but not in
  $H_* MU$}. The first such relation appears in the list above: it is
the relation $0 = Q^8(\xi_1^2) + \xi_1^4 Q^4(\xi_1^2)$ that proves
Hu--Kriz--May's nonsplitting result at the prime $2$. The nonlinear
nature of this relation forces much of the rest of the mess.

The goal of this paper is to describe the obstruction-theoretic method
we used, built to find secondary operations that might genuinely
produce a problem. Our method is based on an attempt to repair Kriz's
paper \cite{kriz-towers}, which circulated quietly in preprint form
but was not published: it assumed an interaction between Steenrod
operations and certain operations in topological Andr\'e--Quillen
cohomology that could not be verified. Trying at length to understand
the exact interaction between these operations, and how the
obstruction theory could be enhanced to one that took this interaction
into account, led to the direction we discuss in this paper. It is
difficult to emphasize adequately how much debt we owe to Kriz's work.

\subsection*{Acknowledgements}

We are thankful to Matthew Ando, Jacob Lurie, Catherine Ray, Birgit
Richter, Andrew Senger, and Dylan Wilson for conversations related to
this work, and in particular for the motivation to expose the
machinery that we've used. 

\section{Postnikov-based obstructions to commutativity}

For context, we will begin by discussing a revisionist version of
Kriz's technique based on Postnikov towers
\cite{kriz-towers,basterra-andrequillen}.

In rough, for a commutative (or $E_\infty$) ring spectrum $R$ and an $R$-module $M$,
Basterra constructed topological Andr\'e--Quillen cohomology groups
\[
  \TAQ^*(R,M),
\]
by analogy with the Andr\'e--Quillen cohomology groups of ordinary
commutative rings \cite{quillen-ring-cohomology}. There is a
restriction map
\[
  \TAQ^*(R,M) \to [R,\Sigma^* M]
\]
from $\TAQ$-cohomology with coefficients in $M$ to ordinary cohomology
with coefficients in $M$. Just as a connective spectrum $X$ has a
Postnikov tower $\{\dots \to P_2 X \to P_1 X \to P_0 X\}$ of spectra which
is determined by $k$-invariants
\[
  k_n \in [P_n X, \Sigma^{n+2} H\pi_{n+1} X],
\]
a connective ring spectrum $R$ has a Postnikov tower $\{\dots \to P_2 R
\to P_1 R \to P_0 R\}$ of commutative ring spectra which is determined
by $k$-invariants
\[
  \tilde k_n \in \TAQ^{n+2}(P_n R, H\pi_{n+1} R)
\]
that restrict to the $k$-invariants of the underlying spectrum. Thus,
if we have the ability to show that all the $k$-invariants of a
homotopy ring spectrum $R$ lift from cohomology to $\TAQ$-cohomology,
we can lift $R$ from a spectrum to an $E_\infty$ ring spectrum.

Of course, this method is not particularly useful if we can't
calculate anything about $\TAQ$. Kriz also developed a spectral
sequence for calculating $\TAQ$-cohomology, in particular with
coefficients in $\mb F_p$, based on the Miller spectral sequence
\cite{miller-delooping}. Kriz's spectral sequence can be interpreted
as being of the form
\[
  \Der^s_{\mathcal{A}}(H_* R, \Omega^t \mb F_p) \Rightarrow
  \TAQ^{s-t}(R,H\mb F_p),
\]
where the groups $\Der^s_\mathcal{A}$ are Andr\'e--Quillen cohomology
groups in a certain category $\mathcal{A}$ of simplicial
graded-commutative rings equipped with Dyer--Lashof operations. From
knowledge of $\TAQ^{*}(R, H\mb F_p)$, one can deduce $\TAQ^{*}(R,
HM)$ for a $\pi_0(R)$-module $M$ using Bockstein spectral sequences.

One of the main difficulties with the Postnikov-based technique is
essentially the same as the difficulty that occurs when using Serre's
Postnikov-based technique for computing homotopy groups of spheres. To
use the Postnikov technique, you use the coaction of the dual Steenrod
algebra $A_*$ on $H_* R$ to compute the homotopy groups $\pi_* R$, and
then use the Dyer--Lashof operations on $H_* R$ to compute the
$\TAQ$-cohomology groups $\TAQ^*(R, \pi_* R)$ where $k$-invariants
live. This division into homotopy and homology uses knowledge of the
Steenrod coaction and Dyer--Lashof action on the homology groups
separately, but it does not track the fact that there are strong and
concrete relationships between them.

\section{Background: Goerss--Hopkins obstruction theory}

In \cite{goerss-hopkins-moduliproblems}, Goerss and Hopkins developed
a very general obstruction theory for the construction of algebras over
operads. In this section, we will give a brief discussion of their
method.

The Goerss--Hopkins obstruction theory requires two main ingredients.
\begin{itemize}
\item We need a multiplicative homology theory $E_*$, represented by a
  ring spectrum $E$.
\item We need a simplicial operad $\mathcal{O}_\bullet$.
\end{itemize}
These are required to satisfy two main constraints: an Adams--Atiyah
condition \cite[1.4.1]{goerss-hopkins-moduliproblems}, and a
``homotopically adapted'' condition
\cite[1.4.16]{goerss-hopkins-moduliproblems}. Some of the main
consequences are the following.
\begin{itemize}
\item The ring $E_* E$ is flat over $E_*$. This ensures that the pair
  $(E_*, E_* E)$ form a Hopf algebroid and that $E_* X$ naturally
  takes values in its category of comodules.
\item There is a large library of ``basic cells'' $\{X_\alpha\}$ with
  $X_\alpha$ finite and $E_* X_\alpha$ projective over $E_*$. This is
  large enough so that for any spectrum $Y$, there are enough maps
  $X_\alpha \to Y$ so that the images of the maps
  $E_* X_\alpha \to E_* Y$ are jointly surjective.
\item If we take a simplicial $\mathcal{O}_\bullet$-algebra and apply
  $E$-homology, the resulting simplicial $E_*E$-comodule lands in some
  algebraic category $\mathcal{A}$ of simplicial $E_*E$-comodules with
  extra structure.\footnote{Namely, $\mathcal{A}$ is a category of
    algebras over some simplicial monad.}
\item If we take a set of basic cells $X_\alpha$ and form the free
  simplicial $\mathcal{O}_\bullet$-algebra on them, on $E$-homology
  we must get the free object on $E_* X_\alpha$ in our algebraic category
  $\mathcal{A}$.
\end{itemize}

With these ingredients, they built an obstruction theory based on
Dwyer--Kan--Stover's resolution model categories
\cite{dwyerkanstover-e2,bousfield-cosimplicial}. This asks whether we
can construct an algebra over the geometric realization
$|\mathcal{O}_\bullet|$ whose $E$-homology is a prescribed algebra
$B_*$ over $\pi_0 (E_* \mathcal{O}_\bullet)$, and produces
obstructions to it. These live in obstruction groups: $\Ext$ groups
calculated in the homotopical category $\mathcal{A}$. In rough, one
tries to build a simplicial $\mathcal{O}_\bullet$-algebra $X_\bullet$,
one degree at a time, so that the chain complex $E_*(X_\bullet)$ is a
resolution of $B_*$. The $\Ext$-groups in question involve
coefficients, and part of determining the algebra in this obstruction
theory is determining what type of object $\Ext$ takes coefficients
in.

Here are some specializations of the Goerss--Hopkins obstruction theory.
\begin{enumerate}
\item We could let $\mathcal{O}_\bullet$ be the trivial operad, and
  $E = \mb S$ so that $E_*(X) = \pi_*(X)$. A simplicial
  $\mathcal{O}_{\bullet}$-algebra is simply a simplicial spectrum, our
  ``basic cells'' are the shifts $\Sigma^t \mb S$, and the algebraic
  category $\mathcal{A}$ is the category of simplicial modules over
  the stable homotopy groups of spheres $\pi_* \mb S$---which is
  equivalent, via the Dold--Kan correspondence, to chain complexes of
  $\pi_* \mb S$-modules. The $\Ext$-groups here are classical
  $\Ext$-groups in $\pi_* \mb S$-modules. Given a $\pi_* \mb S$-module
  $M_*$, the Goerss--Hopkins obstruction theory for the existence of a
  spectrum $X$ with $\pi_* X \cong M_*$ takes place in the groups
  \[
    \Ext^{s}_{\pi_* \mb S}(M_*, \Omega^{t} M_*).
  \]
  Specifically, obstructions to existence occur when $t-s = -2$, and
  to uniqueness occur when $t-s = -1$.

  The first obstruction occurs in $\Ext^3_{\pi_* \mb S}(M_*,
  \Omega M_*)$, and we can be relatively explicit about it. Start
  with a free resolution
  \[
    0 \leftarrow M_* \leftarrow \oplus \Sigma^{m_i} \pi_* \mb S
    \xleftarrow{d_1} \oplus \Sigma^{n_j} \pi_* \mb S \leftarrow \dots
  \]
  and lift $d_1$ to a map of spectra $\bigvee \Sigma^{m_i} \mb S \leftarrow
  \bigvee \Sigma^{n_j} \mb S$, with cofiber $X^{(1)}$. This cofiber
  $X^{(1)}$ is our first approximation, and there is an exact sequence
  \[
    0 \to M_* \to \pi_* X^{(1)} \to \Sigma \ker(d_1) \to 0.
  \]
  If $X^{(1)}$ is going to eventually map to a spectrum $X$ so that
  $M_* \subset \pi_* X^{(1)}$ maps isomorphically to $\pi_* X$, this
  short exact sequence must be split over $\pi_* \mb S$. Thus, a first
  obstruction takes place in the group
  \[
    \Ext_{\pi_* \mb S}^1(\Sigma\ker(d_1), M_*) \cong \Ext_{\pi_* \mb
      S}^1(\ker(d_1),\Omega M_*).
  \]
  Applying the isomorphisms on $\Ext$ induced by the exact sequences
  \begin{align*}
    0 \to \ker(d_1) \to \oplus \Sigma^{n_j} \pi_* \mb S \to \im(d_1)
    \to 0 &\text{ and}\\
    0 \to \im(d_1) \to \oplus \Sigma^{m_i} \pi_* \mb S \to M_*
    \to 0
  \end{align*}
  identifies this with an obstruction in $\Ext^3_{\pi_* \mb S}(M_*,
  \Omega M_*)$.

  If the sequence is split, then we use the splitting to construct a
  map $\bigvee \Sigma^{p_k+1} \mb S \to X^{(1)}$ whose image in
  homotopy groups is the complementary summand, take the cofiber
  $X^{(2)}$, and again examine the resulting homotopy groups to get a
  splitting obstruction in $\Ext^4_{\pi_* \mb S}(M_*, \Omega^2
  M_*)$. This process then continues.\footnote{We note that there is
    absolutely nothing special about the stable homotopy category
    here, and we could carry this procedure out in the category of
    modules over an associative ring spectrum, the category of modules
    over a differential graded algebra, and many other triangulated
    categories that have generators.}
\item Again letting $\mathcal{O}_\bullet$ be the trivial operad, we
  could let $E$ be the Eilenberg--Mac Lane spectrum $H = H\mb F_p$. A
  simplicial $\mathcal{O}_\bullet$-algebra is a simplicial spectrum,
  our ``basic cells'' are all the finite spectra, and the algebraic
  category $\mathcal{A}$ is the category of simplicial comodules over
  the mod-$p$ dual Steenrod algebra $A_*$---which is equivalent to
  chain complexes of $A_*$-comodules. The $\Ext$-groups here are
  classical $\Ext$-groups in $A_*$-comodules. Given an $A_*$-comodule
  $M_*$, the Goerss--Hopkins obstruction theory for the existence of a
  spectrum $X$ with $H_* X \cong M_*$ as an $A_*$-comodule takes place
  in the groups
  \[
    \Ext^s_{A_*}(M_*, \Omega^t M_*).
  \]
  Specifically, obstructions to existence occur when $t-s = -2$, and
  obstructions to uniqueness occur when $t-s=-1$.
  
  A handy interpretation for uniqueness is as follows: given two such
  spectra $X$ and $Y$, these $\Ext$ groups form the $(-1)$-line of the
  Adams spectral sequence for calculating maps $X \to Y$. If $X$ and
  $Y$ are inequivalent, then any isomorphism $\phi\co H_* X \to H_* Y$
  doesn't lift to a spectrum map, and so the corresponding element
  $\phi \in \Hom_{A_*}(H_* X, H_* Y)$ on the $0$-line must support a
  differential that hits an element on the $-1$-line. Similar
  techniques are commonly used in the calculation of Picard groups
  \cite{hopkins-mahowald-sadofsky}.
\item Complementary to this, we could let $\mathcal{O}_\bullet$ be a
  nontrivial operad: the associative operad (viewed as a constant
  simplicial operad), but return to choosing $E = \mb S$. Then a
  simplicial $\mathcal{O}_\bullet$-algebra is a simplicial object in
  associative ring spectra, our basic cells are shifts of the sphere,
  and the algebraic category $\mathcal{A}$ is the category of
  simplicial algebras over $\pi_* \mb S$.\footnote{Proving that this
    is homotopically adapted is now more work, and fails if we try to
    replace ``associative'' with ``commutative'' because the homotopy
    groups of the free commutative ring spectrum on $X$ are a more
    confusing functor of $\pi_* X$.} The $\Ext$-groups of a
  $\pi_* \mb S$ algebra $B_*$ here have coefficients in a
  $B_*$-bimodule. Given a $\pi_* \mb S$-algebra $B_*$, the
  Goerss--Hopkins obstruction theory for the existence of an
  associative ring spectrum $R$ with $\pi_* R \cong B_*$ (as rings)
  takes place in certain groups
  \[
    \Der^s_{\text{assoc}}(B_*, \Omega^t B_*).
  \]
  Again, obstructions to existence occur when $t-s = -2$, and
  obstructions to uniqueness occur when $t-s=-1$. These are often
  called Andr\'e--Quillen cohomology groups of an associative
  algebra with coefficients in a bimodule
  \cite{quillen-ring-cohomology}, and are the nonabelian derived
  functors of derivations. They are also closely related to Hochschild
  cohomology: there is an exact sequence
  \[
    0 \to HH^0(B_*, M_*) \to M_* \to \Der^0(B_*, M_*) \to HH^1(B_*,
    M_*) \to 0
  \]
  that identifies the first two Hochschild cohomology groups with
  central elements and derivations modulo principal derivations, and
  there are isomorphisms
  \[
    \Der^s(B_*, M_*) \to HH^{s+1}(B_*, M_*)
  \]
  for $s > 0$.\footnote{Again, there is nothing special about $\mb
    S$ here, and we could apply this an obstruction theory for
    algebras over a commutative ring spectrum $R$ or differential
    graded algebras over a commutative ring.}
\item We could mix these procedures, getting an obstruction theory
  for associative ring spectra based on mod-$p$ homology that lives
  in Andr\'e--Quillen cohomology groups---for algebras in the category
  of $A_*$-comodules.
\item We should mention, at least in passing, the possibility of using
  a nonconstant operad $\mathcal{O}_\bullet$---for example,
  $\mathcal{O}_\bullet$ could be a simplicial resolution of the
  commutative operad. In the commutative case this tends to lead to an
  obstruction theory closely related to Robinson's obstruction theory,
  whose obstruction groups are $\Gamma$-cohomology groups
  \cite{robinson-gammahomology}. These obstruction groups have been
  examined in detail for $BP$ in \cite{richter-BPcoherences}, and do
  not take the Dyer--Lashof operations into account. Part of the goal
  of this paper is to develop an obstruction theory which does.
\end{enumerate}

\begin{rmk}
  The reader might wonder why we even bother to mention cohomology
  groups other than those containing obstructions. It is worth
  pointing out that these groups do more: the groups
  \[
    \Ext^{s}(E_* X, \Omega^t E_* Y)
  \]
  serve as a tool for calculating the homotopy groups $\pi_{t-s}
  \Map(X,Y)$ for spaces of maps between two realizations
  \cite{bousfield-cosimplicial}. In the above discussion, these
  specialize to things such as the universal coefficient spectral
  sequence and the Adams spectral sequence.
\end{rmk}

\section{Homology-based obstructions to commutativity}

In this section we will discuss a specialization of the
Goerss--Hopkins obstruction theory developed by Senger
\cite{senger-obstr}. Just as Serre's method is improved to Adams' by
switching from a technique that proceeds one homotopy group at a time
to one that uses all the cohomology information simultaneously, the
Postnikov-based obstruction theory is sometimes improved by the
Goerss--Hopkins method that can use both the Dyer--Lashof and Steenrod
information simultaneously.

To set up this obstruction theory, we need a simplicial operad
$\mathcal{O}_\bullet$ (which we choose to be a constant
$E_\infty$-operad) and a homology theory (which we choose to be
mod-$p$ homology $H_*$). A simplicial $\mathcal{O}_\bullet$-algebra is
then a simplicial $E_\infty$ ring spectrum, and our ``basic cells''
are free algebras on finite spectra. Mod-$p$ homology satisfies the
Adams--Atiyah condition, and the fact that this operad is
homotopically adapted amounts to the following theorems.
\begin{thm}[{\cite[\S III.1]{bmms-hinfty}}]
  For an $E_\infty$ ring spectrum $R$, the mod-$p$ homology $H_* R$
  has the following structure.
  \begin{enumerate}
  \item It is a comodule over the mod-$p$ dual Steenrod algebra.
  \item It is a graded-commutative ring.
  \item It has Dyer--Lashof operations that satisfy the Cartan
    formula, Adem relations, and instability relations.
  \item It satisfies the Nishida relations.
  \end{enumerate}
\end{thm}

Following the literature, we will refer to such algebras as
$AR$-algebras.\footnote{Because the origin of this subject is in
  studying the homology of infinite loop spaces rather than the
  homology of $E_\infty$ ring spectra, the definitions of
  $AR$-algebras in the literature often involve connectivity
  assumptions and only discuss Dyer--Lashof operations of nonnegative
  degree.}
  
\begin{thm}[{\cite[\S IX.2]{bmms-hinfty}}]
  The following results about $AR$-algebras hold:
  \begin{enumerate}
  \item The forgetful functor from $AR$-algebras to graded comodules has
    a left adjoint $\mb Q$.
  \item If a comodule $M$ has a basis over $\mb F_p$ of
    elements $e_i$ in degrees $n_i$, then $\mb Q(M)$ is a free
    graded-commutative algebra algebra on the elements $P(e_i)$ such
    that $P$ is an admissible monomial in the Dyer--Lashof algebra of
    excess at least $n_i$.
  \item If we write $\mb P(X)$ for the free $E_\infty$ ring spectrum
    on $X$, then the homology of $\mb P(X)$ is a free $AR$-algebra: the
    natural map $H_* X \to H_* \mb P(X)$ induces a natural isomorphism
    \[
      \mb Q(H_* X) \to H_* \mb P(X).
    \]
  \end{enumerate}
\end{thm}

As a result, the mod-$p$ homology of simplicial $E_\infty$ ring spectra
takes place in the category of simplicial $AR$-algebras. The
$\Ext$-groups of an $AR$-algebra algebra $B_*$ here have coefficients
in an $AR$-$B_*$-module: a $B_*$-module in $A_*$-comodules, with
compatible Dyer--Lashof operations $Q^s$ such that $Q^s(x) = 0$ for $s
\leq |x|$.

\begin{thm}[{\cite{senger-obstr}}]
Given $B_*$, an $AR$-algebra, there
are Goerss--Hopkins obstruction groups
\[
  \Der^{s}_{AR}(B_*, \Omega^t B_*)
\]
calculated in the category of simplicial $AR$-algebras. The groups
with $t-s = -2$ contain an iterative sequence of obstructions to
realizing $B_*$ by an $E_\infty$ ring spectrum $R$ such that
$H_* R \cong B_*$, and the groups with $t-s = -1$ contain obstructions
to uniqueness.
\end{thm}

From this point forward, it will be our goal to give methods to both
calculate these $\Ext$-groups in specific cases and to interpret
elements in them as concrete obstructions.

\section{Tools for calculation}

In this section, we will begin discussing how \cite{senger-obstr}
reduces the Goerss--Hopkins obstruction theory for the Brown--Peterson
spectrum $BP$ and its truncated versions $BP\langle n\rangle$ to more
straightfoward calculations.\footnote{Because these calculations only
  depend on mod-$p$ homology, they apply to the generalized
  $BP\langle n\rangle$ as defined in \cite[\S 3]{level3}.} We need to
calculate the obstruction groups
\[
  \Der^s_{AR}(H_* BP, \Omega^t H_* BP).\footnote{From here forward,
    rather than using the shift operator $\Omega^t$ we will often
    regard the obstruction groups as bigraded. The number $s$ is the
    \emph{filtration} degree, and the number $t-s$ is the \emph{total}
  degree.}
\]

We begin by recalling the structure of $H_* BP$ as a comodule over the
dual Steenrod algebra. The dual Steenrod algebra $A_*$ has quotient Hopf
algebras, given by exterior algebras
\[
  E(n)_* = \Lambda[\tau_0, \tau_1, \dots, \tau_n]
\]
with $|\tau_i| = 2p^i - 1$. (At the prime $2$, $\tau_i$ is the image
of $\xi_i$.) When $n=\infty$, we get a Hopf algebra $E_*$. The
homology $H_* BP\langle n\rangle$ can be identified with a coextended
comodule:
\[
  H_* BP\langle n\rangle \cong A_* \square_{E(n)_*} \mb F_p.
\]
This coextension functor is right adjoint to the functor from
$A_*$-comodules to $E(n)_*$-comodules. To proceed, we need to know that
this adjunction is compatible with Dyer--Lashof operations.
\begin{prop}
  Let $p=2$, and let $M_r$ be the Milnor primitive of
  degree $2^r - 1$ dual to $\xi_r$ in the Milnor basis of the dual
  Steenrod algebra. Then the Nishida relations imply
  \begin{equation}
    \label{eq:milnornishida}
    M_r Q^s = (s+1)Q^{s - 2^r + 1} + \sum_{0 \leq k < r} Q^{s - 2^r + 2^k}
    M_k.
  \end{equation}
\end{prop}
The existence of this formula means that it is possible to define a
category of $E(n)_*$-comodule algebras or $E_*$-comodule algebras with
Dyer--Lashof operations, which we will refer to as $E(n)R$-algebras or
$ER$-algebras respectively. These are compatible across $n$.

\begin{prop}
  The forgetful functors from $A_*$-comodules to $E_*$-comodules and
  $E(n)_*$-comodules lift to ones from $AR$-algebras to $ER$-algebras and
  $E(n)R$-algebras. These functors have exact right adjoints, given by
  $A_* \square_{E(n)_*} (-)$ and $A \square_{E_*}(-)$ on the underlying
  comodules. 
\end{prop}

\begin{rmk}
  There are analogous equations to \eqref{eq:milnornishida} at odd
  primes:
  \begin{align*}
    M_r Q^s &= \pm \beta Q^{s - 2p^r + 1} \pm \sum_{0 \leq k < r} Q^{s
              - 2p^r + 2p^k} M_k,\\
    M_r \beta Q^s &= \pm \sum_{0 \leq k < r} \beta Q^{s - 2p^r + 2p^k} M_k,\\
  \end{align*}
  We have not verified the signs ourselves and so are not comfortable
  stating them here.
\end{rmk}

Due to exactness of both functors in this free-forgetful adjunction,
the adjunction carries forward to an isomorphism of $\Ext$-groups.

\begin{prop}
  For any $AR$-algebra R with an augmentation $R \to A_* \square_{E_*}
  \mb F_p$, there is an isomorphism
  \[
    \Der^s_{AR}(R, A_* \square_{E_*} \mb F_p) \cong
    \Der^s_{ER}(R, \mb F_p).
  \]
  Similarly, for any $AR$-algebra R with an augmentation $R \to A_*
  \square_{E(n)_*} \mb F_p$, there is an isomorphism
  \[
    \Der^s_{AR}(R, A_* \square_{E(n)_*} \mb F_p) \cong
    \Der^s_{E(n)R}(R, \mb F_p).
  \]
\end{prop}

\begin{rmk}
  It's work noting that there are two $ER$-module (or $E(n)R$-module)
  structures on $\mb F_p$: one has $Q^0$ acting as the identity, and
  the other has $Q^0$ acting trivially. The group appearing above is
  the former. Only the latter has \emph{deloopings}: the shifts
  $\Sigma^t \mb F_p$ for $t \geq 0$ must have $Q^0 = 0$ to satisfy the
  instability relations. Therefore, the obstruction groups $\Der^s(R,
  \Omega^t A_s \square_{E_*} \mb F_p)$ do not, a priori, extend to an
  integer grading in $t$. We will later show that, in the range
  of interest, both modules have the same cohomology, and the added
  integer grading is very useful in assembling a systematic calculation.
\end{rmk}

The groups $\Der^s_{E(n)R}(R, \mb F_p)$ can calculated is as follows:
resolve $R$ by a simplicial $E(n)R$-algebra which is free in each
simplicial degree on some free $E(n)_*$-comodule,\footnote{As we will
  see shortly, because $E(n)_*$ is finite-dimensional its comodules
  are equivalent to $E(n)^*$-modules, and so there is an ample supply
  of free comodules. This is more problematic for $A_*$-comodules or
  $E_*$-comodules.} take indecomposables in each simplicial degree,
apply $\Hom_{E(n)R\Mod}(-,\mb F_p)$ in the category of
$E(n)_*$-\emph{comodules} with Dyer--Lashof operations to get a
cosimplicial abelian group, and take the cohomology groups of the
associated cochain complex. This description as a composite allows us
to get a Grothendieck spectral sequence.

\begin{prop}
  For any $E(n)R$-algebra $R$ with an augmentation $R \to \mb F_p$,
  there is a spectral sequence
  \[
    \Ext^p_{E(n)R\Mod}(AQ_q(R), \mb F_p) \Rightarrow \Der^{p+q}_{E(n)R}(R,\mb F_p),
  \]
  where $AQ_*(R)$ are the ordinary Andr\'e--Quillen homology groups of
  $R$ (the nonabelian derived functors of the indecomposables functor
  $Q$).  In particular, if the underlying algebra of $R$ is a free
  graded-commutative algebra, this degenerates to an isomorphism
  \[
    \Ext^s_{E(n)R\Mod}(Q R, \mb F_p) \cong \Der^s_{E(n)R}(R,\mb F_p).
  \]
\end{prop}

In particular, these propositions can be applied to the $BP\langle
n\rangle$.

\begin{cor}
  For any $n \geq m$, there is an isomorphism
  \[
    \Der^s_{AR}(H_* BP\langle n\rangle, H_* BP\langle m\rangle) \cong
    \Ext^s_{E(m)R\Mod}(QH_* BP\langle n\rangle, \mb F_p).
  \]
\end{cor}

For later calculations, it will be helpful to know the primitives in
$H_* BP$.
\begin{prop}
  The $ER$-module $QH_* BP$ has a basis consisting of the elements
  $[\xi_i^2]$ in degree $2^{i+1} - 2$ at the prime $2$, and $[\xi_i]$
  in degree $2p^i - 2$ at odd primes. These are acted on trivially by
  the Milnor primitives. The Dyer--Lashof operations satisfy
  $Q^{2^j - 2^i}[\xi_i^2]) = [\xi_j^2]$ for all $j > i$ at the prime
  $2$.
\end{prop}

\begin{proof}
  The identification of $H_* BP$ as $\mb F_2[\xi_1^2, \xi_2^2, \dots]$
  or $\mb F_p[\xi_1, \xi_2, \dots]$ makes the identification of the
  indecomposables clear. Since the generating classes are all in even
  degrees and the Milnor primitives are of odd degree, the Milnor
  primitives must act trivially. The remaining formula is a
  consequence of the calculations of Steinberger \cite[\S
  III]{bmms-hinfty}.
\end{proof}

\section{Koszul duality}

Now that we have reduced to calculations of $\Ext$ groups in a
category $E(n)R\Mod$ of graded modules, we should examine the
structure on these modules: it is now much simpler. Because
$E(n)_*$ is finite-dimensional, an $E(n)_*$-comodule structure is
precisely the same as an $E(n)^*$-module structure---an action of the
exterior algebra $\Lambda[M_0,\dots,M_n]$ generated by the Milnor
primitives. Thus, an $E(n)R$-module $N$ is almost the same as a
module over a graded ring: the graded ring with generators $M_k$ and
$Q^s$ subject to quadratic relations.

We now begin systematically specializing to the prime $2$, where these
relations take the following form:
\begin{align*}
  Q^r Q^s &= \sum \binom{i-s-1}{2i-r}Q^{r+s-i} Q^i &\text{if }r > 2s,\\
  M_r M_s &= M_s M_r &\text{if }r > s,\\
  M_r^2 &= 0, \\
  M_r Q^s &= (s-1)Q^{s-2^r+1} + \sum_{0 \leq k < r} Q^{s-2^r+2^k}
       M_k.
\end{align*}
However, $E(n)R$-modules satisfy one further \emph{instability
  relation}:
\[
  Q^s x = 0 \text{ if }s \leq |x|.
\]
The relations above mean that the operators on $E(n)R$-modules have a
canonical basis of monomials of the form $M_0^{\epsilon_0}
M_1^{\epsilon_1} \dots M_n^{\epsilon_n} Q^{a_1}\dots Q^{a_m}$, where
$a_i \leq 2a_{i+1}$ and $\epsilon_i \in \{0,1\}$. If $E(n)R$-modules
were actually modules over a graded ring, this graded ring qould be a
quadratic algebra and this basis would be a PBW-basis in the sense of
\cite{priddy-koszul}. Under these circumstances, there would be a Koszul
complex calculating $\Ext$, based on a ``Koszul dual'' quadratic
algebra with differential.

Despite the fact that we are not quite in the case of a quadratic
algebra, Senger showed that Priddy's proof still works. (This
technique was originally carried out for the nonnegative-degree
Dyer--Lashof algebra by Miller \cite{miller-delooping}).

\begin{prop}[{\cite{senger-obstr}}]
  For $N$ an $E(n)R$-module, there is a Koszul complex $C^*(N)$
  calculating $\Ext_{E(n)R\Mod}(N, \mb F_p)$. Let $N$ have basis
  $\{y^i\}$ with dual basis $\{y_i\}$. The Koszul complex $C^* (N)$
  has a basis consisting of monomials
  \[
    \lambda v_0^{k_0} v_1^{k_1} \dots v_n^{k_n} R^{a_1} R^{a_2} \dots
    R^{a_m} y^i
  \]
  where $k_i \geq 0$, $a_i \geq 2a_{i+1}$, and $-|y^i| + 2 \leq
  a_m$. We refer to such monomials as \emph{admissible}.
  Here the operators $v_i$ of total degree $2^{i+1}-2$ are dual
  to $M_i$, $R^a$ of total degree $-a$ is dual to $Q^{a-1}$, and $\lambda$
  of total degree $0$ is dual to the unit of $\mb F_2$. These are
  subject to the following relations:
  \begin{align*}
    R^a R^b &= \sum \binom{b-1-c}{a-2c} R^{a+b-c} R^c \text{ if }a <
              2b\\
    R^a v_i &= \sum_{i < k \leq n} v_k R^{a - 2^i + 2^k}\\
    v_i v_j &= v_j v_i
  \end{align*}
  The operators $R^a$ are also subject to the \emph{instability
    constraint}: the operator $R^a$ can only be applied to an element
  $z$ if $-|z| < a+1$.

  The differential in the Koszul complex is determined by relations
  \begin{align*}
    d R^a(x) &= (a+1)\sum_{0 \leq k \leq n} v_k R^{a + 2^k - 1}(x) + R^a(dx) \text{ and}\\
    d v_i x &= v_i dx,\\
    d \lambda x &= \begin{cases}
      \lambda R^1 x + \lambda(dx)&\text{if }|x| > 0,\\
      \lambda (dx) &\text{if }|x| \leq 0,
    \end{cases}
  \end{align*}
  and the fact that the differential on the $y^i$ is dual to
  the action of the $Q^s$ and $M_k$ on $N$.
\end{prop}

The element $\lambda$, with its differential, appears precisely due to
the fact that $Q^0$ acts by the identity on the coefficient group $\mb
F_p$. However, because $\lambda$ preserves the differential
except on classes in total degree $t-s > 0$, it does not alter the
Goerss--Hopkins obstrunction groups that are important to us.

\begin{thm}
  For $N$ an $E(n)R$-module, the groups $\Der^s(N;\mb F_p)$ agree in
  total degree $t-s < 0$ for the two different actions of $Q^0$ on
  $\mb F_p$.
  
  For the zero action on $\mb F_p$, there is a Koszul complex $C^*(N)$
  calculating $\Ext_{E(n)R\Mod}(N, \mb F_p)$. Let $N$ have basis
  $\{y^i\}$ with dual basis $\{y_i\}$. The Koszul complex $C^* (N)$
  has a basis consisting of monomials
  \[
    v_0^{k_0} v_1^{k_1} \dots v_n^{k_n} R^{a_1} R^{a_2} \dots
    R^{a_m} y^i
  \]
  where $k_i \geq 0$, $a_i \geq 2a_{i+1}$, and $-|y^i| + 2 \leq
  a_m$. We refer to such monomials as \emph{admissible}.
  
  These are subject to the following relations:
  \begin{align*}
    R^a R^b &= \sum \binom{b-1-c}{a-2c} R^{a+b-c} R^c \text{ if }a <
              2b\\
    R^a v_i &= \sum_{i < k \leq n} v_k R^{a - 2^i + 2^k}\\
    v_i v_j &= v_j v_i
  \end{align*}
  The operators $R^a$ are also subject to the \emph{instability
    constraint}: the operator $R^a$ can only be applied to an element
  $z$ if $-|z| < a+1$.

  The differential in the Koszul complex is determined by relations
  \begin{align*}
    d R^a(x) &= (a+1)\sum_{0 \leq k \leq n} v_k R^{a + 2^k - 1}(x) + R^a(dx) \text{ and}\\
    d v_i x &= v_i dx
  \end{align*}
  and the fact that the differential on the $y^i$ is dual to
  the action of the $Q^s$ and $M_k$ on $N$.
\end{thm}

This gives us a large but \emph{explicit} cochain complex that
calculates our Goerss--Hopkins obstruction groups for the realization
of $BP\langle n\rangle$. The groups where potential obstructions live
are in total grading $-2$.

\section{Filtrations and stability}

The Koszul complex calculating these obstruction groups for
$BP\langle n\rangle$ is relatively large and involves complicated
interaction between the $R^a$ and $v_i$. In addition, the obstruction
groups that we described in the previous section depend on $n$ very
strongly. This is part and parcel of how we're working: we're using
the Adams filtration rather than the Postnikov filtration. For any two
different values of $n$, the Adams towers for $BP\langle n\rangle$ are
quite different as soon as one reaches filtration $1$, and so our
multiplicative obstruction theory doesn't really stabilize as $n$
grows.

Fortunately, both of these problems can be addressed to some
degree. There are several natural filtrations on this complex obtained
by assigning degrees to the $v_i$, and this allows us to calculate by
an inductive method.

\begin{prop}
  For any $0 \leq k \leq n$ and any $E(n)R$-module $N$, the Koszul
  complex has quotient complexes
  \[
    D^{k} = C^*(N) / (v_{k+1}, \dots, v_n)
  \]
  and for $0 \leq k < n$ there are Bockstein spectral sequences
  \[
    H^*(D^{k}) \otimes \mb F_p[ v_{k+1},\dots,v_m ] \Rightarrow H^*(D^m).
  \]
\end{prop}

In terms of $\Ext$-groups, these have concrete interpretations. For
any $k \leq n$ and any $E(n)R$-module $N$, we can view $N$ as an
$E(k)R$-module, and $D^{k}$ is a Koszul complex calculating $\Ext$ in
$E(k)R$-modules. Thus, this can be viewed as a collection of Bockstein
spectral sequences
\[
  \Ext_{E(k)R\Mod}(N,\mb F_p)\otimes \mb F_p[v_{k+1},\dots,v_m] \Rightarrow
  \Ext_{E(m)R\Mod}(N,\mb F_p).
\]

We now specialize to what happens when we consider indecomposables.

Consider the sequence of maps
\[
  H_* BP \to \dots \to H_* BP\langle 2\rangle \to H_* BP\langle
  1\rangle \to H_* BP\langle 0\rangle \to H_* BP\langle -1\rangle.
\]
This creates an array of Ext-groups:
\[
  \xymatrix{
    \Ext_{AR}(H_* BP\langle -1\rangle, H_* BP\langle -1\rangle) \ar[d]\\
    \Ext_{AR}(H_* BP\langle 0\rangle, H_* BP\langle -1\rangle) \ar[d]&
    \Ext_{AR}(H_* BP\langle 0\rangle, H_* BP\langle 0\rangle) \ar[l] \ar[d]\\
    \Ext_{AR}(H_* BP\langle 1\rangle, H_* BP\langle -1\rangle) \ar[d]&
    \Ext_{AR}(H_* BP\langle 1\rangle, H_* BP\langle 0\rangle) \ar[l] \ar[d]&
    \Ext_{AR}(H_* BP\langle 1\rangle, H_* BP\langle 1\rangle) \ar[l] \ar[d]\\
    \vdots & \vdots & \vdots \\
  }
\]
After our chain of isomorphisms, these can be re-identified:
\[
  \xymatrix{
    \Ext_{E(-1)R\Mod}(QH_* BP\langle -1\rangle,\mb F_p) \ar[d]\\
    \Ext_{E(-1)R\Mod}(QH_* BP\langle 0\rangle, \mb F_p) \ar[d]&
    \Ext_{E(0)R\Mod}(QH_* BP\langle 0\rangle, \mb F_p) \ar[l] \ar[d]\\
    \Ext_{E(-1)R\Mod}(QH_* BP\langle 1\rangle, \mb F_p) \ar[d]&
    \Ext_{E(0)R\Mod}(QH_* BP\langle 1\rangle, \mb F_p) \ar[l] \ar[d]&
    \Ext_{E(1)R\Mod}(QH_* BP\langle 1\rangle, \mb F_p) \ar[l] \ar[d]\\
    \vdots & \vdots & \vdots
  }
\]
Each Bockstein spectral sequence requires one of the terms in this
diagram and converges to the one immediately to its right. Because the
degrees of the elements in the Koszul complex in any column involve
only a finite list $v_0,\dots,v_n$ of positive-degree operators, and
$QH_* BP\langle n\rangle\to QH_* BP\langle m\rangle$ is always an
isomorphism in large degrees, the vertical towers \emph{do} stabilize
to the tower of groups
\[
  \Ext_{E(-1)R\Mod}(QH_* BP, \mb F_p) \leftarrow
  \Ext_{E(0)R\Mod}(QH_* BP, \mb F_p) \leftarrow
  \Ext_{E(1)R\Mod}(QH_* BP, \mb F_p) \leftarrow \dots,
\]
or equivalently the tower
\[
  \Ext_{AR}(H_* BP, H_* BP\langle -1\rangle) \leftarrow
    \Ext_{AR}(H_* BP, H_* BP\langle 0\rangle)\leftarrow
    \Ext_{AR}(H_* BP, H_* BP\langle 1\rangle) \leftarrow \dots
\]
whose limit is roughly $\Ext_{AR}(H_* BP, H_* BP)$.\footnote{Modulo a
  $\lim^1$-issue.}

Since we're interested in $BP$ anyway, computing in this grid gives as
a workaround for Koszul complex problems and convergence problems with
$ER$-algebras that don't show up for $E(n)R$-algebras. In practice
this means that we will be calculating
\[
  \Ext_{AR}(H_* BP, H_* BP\langle n\rangle) \cong \Ext_{E(n)R\Mod}(Q
  H_* BP, \mb F_p).
\]
and calculating Bockstein spectral sequences
\[
  \Ext_{AR}(H_* BP, H_* BP\langle n\rangle) \otimes \mb F_p[v_{n+1}] \Rightarrow 
  \Ext_{AR}(H_* BP, H_* BP\langle n+1\rangle)
\]
to get at the limiting value.

\section{The critical group and secondary operations}

We now consider the critical group: the first Goerss--Hopkins
obstruction group
\[
  \Ext^3_{E(n)R\Mod}(QH_* BP,\Omega \mb F_2)
\]
that could support an obstruction class. The Koszul complex, in this
degree, had a basis of those monomials of the form
\[
  v_i v_j v_k y_m, v_i v_j R^a y_m, v_i R^a R^b y_m, \text{ and }R^a
  R^b R^c y_m
\]
of total degree $-2$, where $y_k$ is dual to $[\xi_k^2] \in QH_*
BP$. Our calculations in the remainder of this paper will determine
exactly what has survived to $\Ext$ and what has not. The first type
supports a differential if $m > 1$ and is usually in degree greater
than $-2$ if $m = 1$; the second type only survives if it is in odd
total degree; the fourth type is always in a large negative
degree. This leaves us only with the third type to carefully check.
\begin{thm}
  The first obstruction group
  \[
    \Ext^3_{E(n)R\Mod}(QH_* BP,\Omega \mb F_2)
  \]
  has a basis of classes of the following forms: the class $v_0^3 y_1$,
  and those of form $v_i R^a R^b y_1$ where $i \geq 3$, $a$ and $b$
  are odd, $b \geq 7$, $a > 2b$, and $a+b = 2^{i+1} - 2$.
\end{thm}

\begin{cor}
  The minimal value of $i$ such that an admissible monomial $v_i R^a
  R^b y_1$ appears is when $i=4$; the only monomials of this type are
  $v_4 R^{23} R^7 y_1$ and $v_4 R^{21} R^9 y_1$.
\end{cor}

It falls to us now to determine what these obstructions mean. In the
Koszul complex, these are basis elements that detect the elements
$[M_4 | Q^{22} | Q^6 | \xi_1^2]$ and $[M_4 | Q^{20} | Q^8 | \xi_1^2]$
in a bar resolution for $QH_* BP$ in $E(n)R\Mod$. How do we interpret
these?

It it easiest to describe these in a constructive fashion. Recall that
the Goerss--Hopkins method is trying to construct $BP$ or $BP\langle
n\rangle$ by starting with the sphere $\mb S$ and iteratively
killing off maps from finite complexes $Z$. Let us examine how the
first obstruction might detect something:
\begin{enumerate}
\item We start with the sphere $\mb S$, whose homology is $\mb F_2$,
  and call it $R_0$.
\item The homology is not yet correct: it is missing $\xi_1^2$, which
  is connected to the unit by the Steenrod operation $\Sq^2$. We
  attach an $E_\infty$-cell in a way that produces it. In this case,
  we can do this by coning off the map $\eta\co S^1 \to \mb S$,
  forming the pushout of $\mb P(CS^1) \leftarrow \mb P(S^1) \to \mb S$
  in $E_\infty$ ring spectra to construct $R_1$.
\item We are no longer missing any homology classes, but the homology
  is still not correct: there are relations that are not yet
  satisfied. There is a relation involving $Q^6$ (this relation
  turns out to be $Q^6(\xi_1^2) = \xi_1^8$ in $H_8 BP$), and we
  know that $Q^n(\xi_1^2) = 0$ for $n$ odd, and so on. We, again, can
  cone off some finite complex in a way that imposes these relations
  on homology. Possibly the element $Q^6(\xi_1^2) + \xi_1^8$ lifts to
  the homotopy of $R_1$ in which case we can cone off a map
  $S^8 \to R_1$, but more likely we can't (e.g. if this class supports
  Steenrod operations) and we need to map in a finite complex
  $Z \to R_1$ that hits this class in homology, and the Steenrod
  operations on it. We construct the pushout
  $\mb P(CZ) \leftarrow \mb P(Z) \to R_1$ to cone this map off and
  construct $R_2$.
\item We have correctly imposed all the relations that should hold
  now, but the homology is still not correct: new classes have
  appeared in homology when we imposed these relations. These come
  from ``relations between relations'' and are called secondary
  homology operations. For example, there is an Adem relation saying
  $Q^{22} Q^6(y) = Q^{17} Q^{11}(y) + Q^{15} Q^{13}(y)$ for elements
  $y \in H_2$, the Cartan formula implies $Q^{22} (\xi_1^8) = 0$, and
  there are relations $Q^{11}(\xi_1^2) = Q^{13}(\xi_1^2) = 0$. The
  nontriviality of $R^{23} R^7 y_1$ implies that these relations
  between relations glue together into a potential secondary
  operation. This new secondary operation
  $\theta(\xi_1^2) \in H_{31} BP$ needs to be eliminated because
  $H_{31} BP$ is supposed to be zero. We map in a finite complex
  $W \to R_2$ to clamp down on this spurious new class
  $\theta(\xi_1^2)$ (and the Steenrod operations on it), coning it off
  to construct a new ring $R_3$.
\item Now we arrive at our potential problem. When we eliminated
  $\theta(\xi_1^2)$, we were forced to killed off the Steenrod
  operations on it, such as $M_4 \theta(\xi_1^2) \in
  H_0(BP)$. However, it's possible that $M_4 \theta(\xi_1^2)$ is the
  unit $1 \in H_0(BP)$, in which case we have been forced to make $H_*
  R_3$ into the zero ring.
\end{enumerate}

For this particular problem, which would be detected by the
obstruction class $v_4 R^{23} R^7 y_1$, this doesn't really
happen. The reason is not particularly interesting: \emph{all the
  relations we used in this argument also take place in $H_* MU$}, and
we know that $MU$ admits an $E_\infty$ ring structure. This secondary
operation we've written down can't actually satisfy
$M_4 \theta(\xi_1^2) = 1$ or this argument would equally well exclude
the existence of an $E_\infty$ ring structure on $MU$.

Fortunately, we have another potential basis element
$v_4 R^{21} R^9 y_1$, which would detect some kind of problem
involving a relation satisfied by $Q^8 (\xi_1^2)$ in $H_* BP$, the
Adem relation for $Q^{20} Q^8$, and the Milnor primitive $M_4$. It
gives us a place to look, and looking here leads us to the start of
\cite{secondary}.

\section{Calculation setup for $BP$}

It's time to get down to the business of
calculation.\footnote{The calculations we will describe in the
  following sections determine many more of the Goerss--Hopkins
  obstruction groups than the critical group that we need. We have
  this calculation available, and it helps to illuminate the critical
  group by examining the algebraic structure in the large.}
In this section we'll begin the process of calculating $\Ext_{AR}(H_*
BP, H_* BP\langle n\rangle)$ inductively at the prime $2$. This uses
knowledge of the structure of the indecomposables $QH_* BP$ as an
$ER$-module. We begin by writing down the Koszul complexes that
calculate $\Ext$.

\begin{prop}
  Let $y_k$ be dual to $[\xi_k^2] \in QH_* BP$. The Koszul complex
  calculating $\Ext_{E(n)R\Mod}(QH_* BP, \mb F_2)$ has a basis of
  monomials
  \[
    v_0^{k_0} v_1^{k_1} \dots v_n^{g_k} R^{a_1} R^{a_2} \dots R^{a_m} y_k
  \]
  ranging over $k_i$, $a_i$, and $k$ such that $k_i \geq 0$, $k \geq
  1$, $a_i \geq 2 a_{i+1}$, and $2^{k+1} \leq a_m$.

  The differential is determined by
  \begin{align*}
    dy_k &= \sum_{j < k} R^{2^{k+1} - 2^{j+1} + 1} y_j,\\
    d v_i x &= v_i dx,\\
    d R^a(x) &= \begin{cases}
      R^a(dx) &\text{if $a$ is odd,}\\
      \sum_{0 \leq k \leq n} v_k R^{a + 2^k - 1}(x) + R^a(dx)&\text{if
        $a$ is even,}
      \end{cases}\\
  \end{align*}
  and the relation
  \[
    R^a v_i x = \sum_{i < j \leq n} v_j R^{a - 2^i + 2^j} x.
  \]
\end{prop}

We can already draw several conclusions from this calculation.
\begin{cor}
  Any admissible monomials of the form $v_0^{k_0} v_1^{k_1} \dots
  v_n^{k_n} R^{a_1} R^{a_2} \dots R^{a_m} y_1$, where the
  $a_1,\dots,a_m$ are odd, are permanent cycles in the $E(n)$-Koszul
  complex.
\end{cor}

\begin{cor}
  The Koszul complex has a filtration under ``degree in
  $R$'': if we say that an admissible monomial $v_0^{k_0} v_1^{k_1}
  \dots v_n^{k_n} R^{a_1} R^{a_2} \dots R^{a_m} y_k$ has weight $m$,
  then the Koszul differential $d$ increases weight, and preserves it
  on classes ending in $y_1$.
\end{cor}

\section{First calculations: $n=-1$}

Our base case for calculation is when $n=-1$, where things simplify
greatly: there is no interference from the $v_i$ or differentials on
$R^a$.

\begin{prop}
  The Koszul complex calculating $\Ext_{E(-1)R\Mod}(QH_* BP, \mb F_2)$
  has a basis of monomials
  \[
    R^{a_1} R^{a_2} \dots R^{a_m} y_k
  \]
  ranging over $a_i$ and $k$ such that $k \geq 1$, $a_i \geq 2
  a_{i+1}$, and $2^{k+1} \leq a_m$. 

  The differential is determined by
  \[
    dR^a(x) = R^a(dx)
  \]
  and
  \[
    dy_k = \sum_{j < k} R^{2^{k+1} - 2^{j+1} + 1} y_j.
  \]
\end{prop}
This allows us to start charting things up and doing the work of
calculating the result. The first portion of the Koszul complex
appears in Figure~\ref{fig:neg1}, with operations on $y_1$ indicated
by classes in black.

\begin{figure}
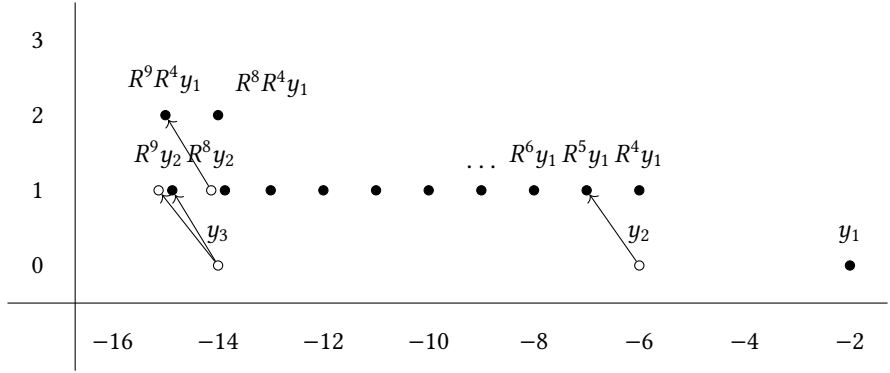

  \centering
\begin{sseqdata}[name=neg1, Adams grading, x tick step = 2, xscale =
    0.70, x range={-16}{-2}, y range={0}{3}]
    \class["y_1" above, fill](-2,0)
    \class["y_2" above](-6,0)
    \class["y_3" above](-14,0)
    \class["R^4 y_1" above, fill](-6,1)
    \class["R^5 y_1" above, fill](-7,1)
    \class["R^6 y_1" above, fill](-8,1)
    \class["\dots" above, fill](-9,1)
    \class[fill](-10,1)
    \class[fill](-11,1)
    \class[fill](-12,1)
    \class[fill](-13,1)
    \class["R^8 y_2" above](-14,1)
    \class[fill](-14,1)
    \class["R^9 y_2" above](-15,1)
    \class[fill](-15,1)
    \class["R^8 R^4 y_1" above right, fill](-14,2)
    \class["R^9 R^4 y_1" above, fill](-15,2)
    \d1(-6,0)
    \d1(-14,0,,1)
    \d1(-14,0,,2) 
    \d1(-14,1)
  \end{sseqdata}
  \printpage[name = neg1, page = 1]
  
  \caption{Part of the Koszul complex for $n=-1$}
  \label{fig:neg1}
\end{figure}
\begin{prop}
  The $\Ext$-groups $\Ext_{E(-1)R\Mod}(QH_* BP, \mb F_2)$ have a basis of
  admissible monomials
  \[
    R^{a_1} R^{a_2} \dots R^{a_m} y_1
  \]
  (meaning, monomials such that $a_i \geq 2a_{i+1}$ and $a_m \geq 4$)
  which do not end in any of the sequences $R^5 y_1$, $R^9 R^4 y_1$,
  $R^{17} R^8 R^4 y_1$, $\dots$
\end{prop}

\begin{rmk}
  An examination of the Adem relations for the operations $R^a$ finds
  that this result means that we essentially have a ``free unstable
  module'' over the $R^a$ subject to one relation: $R^5 y_1 = 0$. For
  instance, the Adem relations imply $R^8 R^5 = R^9 R^4$, $R^{16} R^8
  R^5 = R^{17} R^8 R^4$, and so on.
  
\end{rmk}

\begin{proof}[Proof sketch.]
  We filter the Koszul complex by putting $y_i$ into filtration $i$ so
  that in the associated graded complex the differential becomes
  $d(y_i) = R^{2^i+1} y_{i-1}$. The claim is then that, on the
  associated graded, this is exact except in filtration zero.

  In degree zero of the associated graded, we have everything of the
  form $R^A y_1 = R^{a_1} R^{a_2} \dots R^{a_k} y_1$, and all of those
  are permanent cycles.

  In degree one of the associated graded, we have
  $d(R^A y_2) = R^A R^5 y_1$. (We remark that the elements $R^A y_2$
  are only defined when the last term is $R^8$ or higher.) The map
  $R^A \mapsto R^A R^5$ has image consisting of the right multiples
  of $R^5$. It is also injective on anything where the admissible $R^A$ ends
  with $R^{10}$ or higher, because then $R^A R^5$ is still
  admissible. Thus we only have to see what happens to the admissible
  monomials that end in $R^9$ or $R^8$.

  The Adem relations say $R^9 R^5 = 0$, so all admissible monomials
  ending in $R^9$ are in the kernel.

  The Adem relations also say that $R^8 R^5 = R^9 R^4$. The admissible
  monomials $R^B R^8$ where $B$ ends in $R^{18}$ or higher map
  isomorphically to admissible monomials of the form $R^B R^9
  R^4$. That now just leaves us checking admissible monomials that end
  in $R^{17} R^8$ or $R^{16} R^8$.

  The Adem relations say $R^{17} R^9 = 0$, so anything ending in
  $R^{17} R^8 y_2$ is in the kernel. The Adem relations also say that
  $R^{16} R^9 = R^{17} R^8$. We inductively repeat this pattern.

  We ultimately find that the kernel in grading one consists of any
  admissible monomials $R^A y_2$ where $R^A$ ends in $R^9$, $R^{17}
  R^8$, $R^{33} R^{16} R^8$, $R^{65} R^{32} R^{16} R^8$, and so
  on. These are precisely all the right multiples of $R^9$.

  We then look at grading two, where we have $d(y_3) = R^9 y_2$; all
  the right multiples of $R^9$ are in the image. We run the exact same
  computation and find that the kernel consists of all the multiples
  of $R^{17}$. This procedure continues indefinitely.
\end{proof}

\begin{rmk}
  This is very similar to a computation of the topological
  Andr\'e--Quillen cohomology of $H\mb F_2$ stated in
  \cite{lazarev-ainfty} and determined by alternative means in Hoyer's
  thesis \cite{hoyer-thesis}.
\end{rmk}

The final answer provides us with a bit of relief: the cohomology of
the Koszul complex is actually a lot \emph{less} complicated to
describe than the Koszul complex itself.

\begin{rmk}
  It is tempting to hope that the Koszul complex is quasi-isomorphic
  to its quotient by all the $y_i$ for $i > 1$ and by the relation
  $R^5 y_1 = 0$. Unfortunately, this is not the case. For example, the
  differentials $d(y_2) = R^5 y_1$ and $d(y_3) = R^9 y_2 + R^{13} y_1$
  in the Koszul complex show that the Koszul complex has an identity
  of secondary operations $\langle R^9, R^5, y_1\rangle = R^{13} y_1$ not
  satisfied in the quotient complex.
\end{rmk}

\section{Further calculations: $n=0$}

To calculate the next $\Ext$-groups, we can feed our previous
$\Ext$-calculation into a Bockstein spectral sequence.
\begin{prop}
  The Bockstein spectral sequence 
  \[
    \Ext_{E(-1)R\Mod}(QH_* BP, \mb F_2) \otimes \mb F_2[v_0] \Rightarrow
    \Ext_{E(0)R\Mod}(QH_* BP, \mb F_2)
  \]
  has a basis of admissible monomials
  \[
    v_0^{k} R^{a_1} R^{a_2} \dots R^{a_m} y_1
  \]
  (except for those which end in $R^5$, $R^9 R^4$, \dots) with
  $v_0$-linear differential satisfying
  \[
    dR^a(x) = (a+1) v_0 R^{a+1} x + R^a(dx)
  \]
  and the rule
  \[
    R^a v_0 = 0.
  \]
\end{prop}
In particular, these rules make it easy to apply the differential
to any element in our basis:
\[
  d(v_0^k R^{a_1} R^{a_2} \dots R^{a_m} y_1) =
  \begin{cases}
    0 &\text{if $a_1$ is odd}\\
    v_0^{k+1} R^{a_1+1} R^{a_2} \dots R^{a_m} y_1 &\text{if $a_1$ is even}
  \end{cases}
\]
In essence, our admissible monomials that start with $R^{even}$ are
attempting to make the admissible monomials that start with $R^{odd}$
into $v_0$-torsion elements. Calculating the $d_1$-differential is an
exercise in being careful about edge cases.

\begin{prop}
  The $\Ext$-groups
  \[
    \Ext_{E(0)R\Mod}(QH_* BP, \mb F_2)
  \]
  are a direct sum of copies of two types of terms:
  \begin{enumerate}
  \item $v_0$-torsion copies of $\mb F_2$ indexed by those admissible
    monomials $R^{a_1} R^{a_2} \dots R^{a_m} y_1$ with $a_1$ odd which
    do not end in any of the sequences $R^5 y_1$, $R^9 R^4 y_1$,
    $R^{17} R^8 R^4 y_1$, and so on, and
  \item copies of $\mb F_2[v_0]$ indexed by the admissible monomials
    $y_1$, $R^4 y_1$, $R^8 R^4 y_1$, $R^{16} R^8 R^4 y_1$, and so on.
  \end{enumerate}
\end{prop}

The initial portion of these $\Ext$-groups is sketched in
Figure~\ref{fig:pos0}.
\begin{figure}
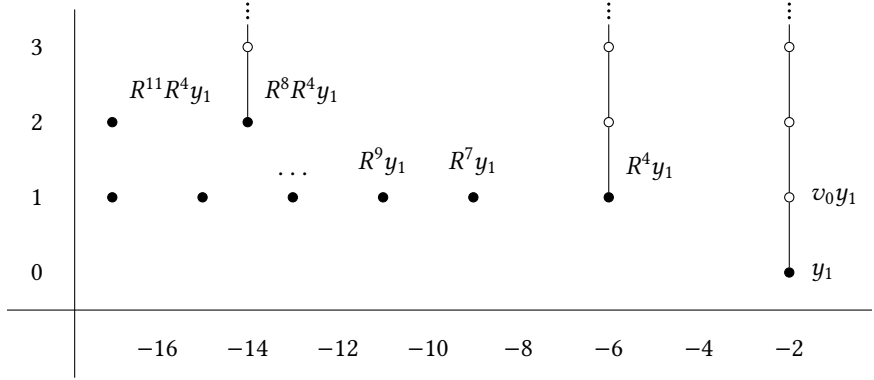

  \centering
\begin{sseqdata}[name=pos0, Adams grading, x tick step = 2, xscale =
    0.6, x range={-17}{-1}, y range={0}{3}]
    \class["y_1" right, fill](-2,0)
    \class["v_0 y_1" right](-2,1)
    \class(-2,2)\class(-2,3)\class(-2,4)
    \structline(-2,0)(-2,1)
    \structline(-2,1)(-2,2)
    \structline(-2,2)(-2,3)
    \structline(-2,3)(-2,4)
    \class["R^4 y_1" above right, fill](-6,1)
    \class(-6,2)\class(-6,3)\class(-6,4)
    \structline(-6,1)(-6,2)
    \structline(-6,2)(-6,3)
    \structline(-6,3)(-6,4)
    \class["R^7 y_1" above, fill](-9,1)
    \class["R^9 y_1" above, fill](-11,1)
    \class["\dots" above, fill](-13,1)
    \class[fill](-15,1) 
    \class[fill](-17,1)
    \class["R^8 R^4 y_1" above right, fill](-14,2)
    \class(-14,3)\class(-14,4)
    \structline(-14,2)(-14,3)
    \structline(-14,3)(-14,4)
    \class["R^{11} R^4 y_1" above right, fill](-17,2)
  \end{sseqdata}
  \printpage[name = pos0, page = 1]
  
  \caption{Part of the $\Ext$-groups for $n=0$}
  \label{fig:pos0}
\end{figure}

\begin{proof}[Proof sketch]
  The first observation is that if $R^{a_1} R^{a_2} \dots R^{a_m}$ is
  an admissible monomial and $a_1$ is odd, then
  $R^{a_1 - 1} R^{a_2} \dots R^{a_m}$ is also an admissible
  monomial. Therefore, every admissible monomial starting with an
  \emph{odd} term is a permanent cycle that becomes annihilated by
  $v_0$.

  Almost all of the admissible monomials
  $R^{a_1} R^{a_2} \dots R^{a_m}$ where $a_1$ is even, by contrast,
  support differentials and do not survive the spectral sequence. The
  only exception is when the admissible monomial
  $R^{a_1+1} R^{a_2} \dots R^{a_m}$ is already zero, and this occurs
  only when it ends with one of the sequences $R^5$, $R^9 R^4$,
  $R^{17} R^8 R^4$, and so on. In this case, there are two
  possibilities: either $R^{a_1} R^{a_2} \dots R^{a_m}$ \emph{also}
  ends with this sequence and it was already the zero monomial, or it
  is one of the monomials $R^4$, $R^8 R^4$, $R^{16} R^8 R^4$, and so
  on. These monomials are permanent cycles and produce infinite
  $v_0$-towers.
\end{proof}

\section{Further calculations: $n=1$}

There are two useful spectral sequences $\Ext_{E(1)R\Mod}(QH_* BP, \mb
F_2)$ based on our previous work. Playing the information in these two
spectral sequences off each other provides a useful technique for
resolving hidden extensions.

The first spectral sequence is the Bockstein spectral sequence
\[
  \Ext_{E(0)R\Mod}(QH_* BP, \mb F_2) \otimes \mb F_2[v_1] \Rightarrow
  \Ext_{E(1)R\Mod}(QH_* BP, \mb F_2)
\]

The second one is the filtration on the Koszul complex that puts both $v_0$
and $v_1$ in filtration $1$ simultaneously:
\[
  \Ext_{E(-1)R\Mod}(QH_* BP, \mb F_2) \otimes \mb F_2[v_0, v_1] \Rightarrow
  \Ext_{E(1)R\Mod}(QH_* BP, \mb F_2)
\]
that has a basis of monomials
\[
  v_0^{k} v_1^{l} R^{a_1} R^{a_2} \dots R^{a_m} y_1
\]
(except for those which are right multiples of $R^5$) with
$\mb F_2[v_0, v_1]$-linear differential satisfying the following
relations:
\begin{align*}
  dR^a(x) &= (a+1) (v_0 R^{a+1} x + v_1 R^{a+3} x) + R^a(dx)\\
  R^a(v_1 x) &= 0\\
  R^a(v_0 x) &= v_1 R^{a+2} x
\end{align*}

In particular, these rules make it mechanical to apply the
differential to any element in our basis:
\[
  d(R^{a_1} R^{a_2} x) =
  \begin{cases}
    0 &\text{if $a_1$ and $a_2$ are odd}\\
    v_1 R^{a_1 + 2} R^{a_2+1} x &\text{if $a_1$ is odd and
    $a_2$ is even}\\
    v_0 R^{a_1 + 1} R^{a_2} x + v_1 R^{a_1 + 3}
    R^{a_2} x &\text{if $a_1$ is even and $a_2$ is odd}\\
    v_0 R^{a_1 + 1} R^{a_2} x + v_1 R^{a_1 + 3}  R^{a_2} x + v_1
    R^{a_1 + 2} R^{a_2+1} x
    &\text{if $a_1$ and $a_2$ are even}
  \end{cases}
\]
We will mostly use this to inform us about differentials and hidden
extensions in the $v_1$-Bockstein spectral sequence. For example, we
obtain the following results:
\begin{lem}
  In the $v_1$-Bockstein spectral sequence, there are
  $d_1$-differentials
  \[
    d(R^{a_1} R^{a_2} R^{a_3} \dots R^{a_m} y_1) = v_1 R^{a_1+2}
    R^{a_2+1} R^{a_3} \dots R^{a_m} y_1
  \]
  when $a_1$ is odd and $a_2$ is even.
\end{lem}

\begin{lem}
  In the $v_1$-Bockstein spectral sequence, when $x$ is a cycle there
  are hidden extensions
  \[
    v_0 R^{a_1} R^{a_2} x = v_1 R^{a_1 + 2} R^{a_2} x
  \]
  when $R^a x$ is admissible and $a$ is odd.
\end{lem}

\begin{proof}
  The differential
  \[
    d(R^{a_1-1} R^{a_2} x) = v_0 R^{a_1} R^{a_2} x + v_1 R^{a_1 + 2}
    R^{a_2} x,
  \]
  for $x$ a cycle in the filtered Koszul complex, establishes this
  relation.
\end{proof}

Just as in the previous section, we can now determine the result of
the $v_1$-Bockstein spectral sequence by a careful examination of edge cases.

\begin{prop}
  The $\Ext$-groups
  \[
    \Ext_{E(1)R\Mod}(QH_* BP, \mb F_2)
  \]
  are a direct sum of copies of three types of terms:
  \begin{enumerate}
  \item a free copy of $\mb F_2[v_0, v_1]$ generated by $y_1$,
  \item $(v_0,v_1)$-torsion copies of $\mb F_2$ indexed by those admissible
    monomials $R^{a_1} R^{a_2} \dots R^{a_m} y_1$ with $a_1$ and $a_2$
    odd which do not end in any of the sequences $R^5 y_1$, $R^9 R^4 y_1$,
    $R^{17} R^8 R^4 y_1$, and so on, and
  \item ``$(v_0,v_1)$ sawtooth'' patterns
    \[
      \mb F_2[v_0,v_1] \{R^{2k+1} x\} / (v_0 R^{2k+1} x = v_1
      R^{2k+3} x)
    \]
    as $R^{2k+1} x$ range over admissible monomials with $x$ in the
    set $y_1, R^4 y_1, R^8 R^4 y_1, \dots$, mod the relations that
    $R^5 y_1 = 0$, $R^9 R^4 y_1 = 0$, and so on.
  \end{enumerate}
\end{prop}

The first sawtooth is pictured in Figure~\ref{fig:pos1}. It has
generators $R^7 y_1$, $R^9 y_1$, $R^{11} y_1$, and so on, with
$v_0 R^7 y_1 = v_1 R^9 y_1$, etc, mod the relation
$v_1 R^7 y_1 = 0$. The second sawtooth has generators $R^{11} R^4
y_1$, $R^{13} R^4 y_1$, and so on.
\begin{figure}
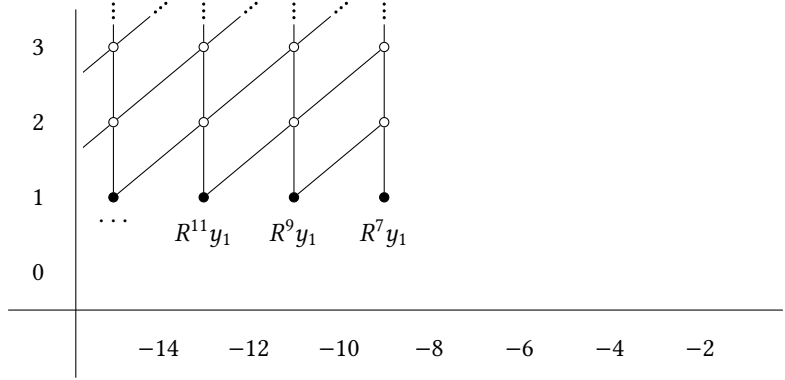

  \centering
\begin{sseqdata}[name=pos1, Adams grading, x tick step = 2, xscale =
    0.6, x range={-15}{-1}, y range={0}{3}]
    \class["R^7 y_1" below, fill](-9,1)
    \class["R^9 y_1" below, fill](-11,1) 
    \class["R^{11} y_1" below, fill](-13,1)
    \class["\dots" below, fill](-15,1) 
    \class[fill](-17,1) 
    \foreach \x in {-9,-11,-13,-15,-17} \foreach \y in {2, 3, 4} {
      \class(\x,\y)
    } 
    \foreach \x in {-9,-11,-13,-15} \foreach \y in {1, 2, 3} {
      \structline(\x,\y)(\x,\y+1)
    } 
    \foreach \x in {-11,-13,-15,-17} \foreach \y in {1, 2, 3} {
      \structline(\x,\y)(\x+2,\y+1)
    }
  \end{sseqdata}
  \printpage[name = pos1, page = 1]
  
  \caption{Part of the first sawtooth pattern in $\Ext$ for $n=1$}
  \label{fig:pos1}
\end{figure}

\begin{proof}[Proof sketch]
  The first observation is that if $R^{a_1} R^{a_2} \dots R^{a_m}$ is
  an admissible monomial and $a_1$ and $a_2$ are odd, then
  $R^{a_1 - 2} R^{a_2-1} \dots R^{a_m}$ is also an admissible
  monomial. Therefore, every admissible monomial starting with two
  \emph{odd} terms has a canonical lower admissible monomial with a
  differential that makes it $v_1$-torsion.

  Almost all of the admissible monomials
  $R^{a_1} R^{a_2} \dots R^{a_m}$ where $a_1$ is odd and $a_2$ is
  even, by contrast, support differentials and do not survive the
  spectral sequence. The only exception is when the admissible
  monomial $R^{a_1+2} R^{a_2+1} \dots R^{a_m}$ is already zero, and this
  occurs only when it ends with one of the sequences $R^5$, $R^9 R^4$,
  $R^{17} R^8 R^4$, and so on. In this case, there are two
  possibilities: either $R^{a_1} R^{a_2} \dots R^{a_m}$ \emph{also}
  ends with this sequence and it was already the zero monomial, or it
  is one of the monomials $R^a$, $R^a R^4$, $R^a R^8 R^4$, $R^{16} R^8
  R^4$, and so on, with $a$ odd.

  The differentials in the Koszul complex imply
  $d(R^4 y_1) = v_1 R^7 y_1$, $d(R^8 R^4 y_1) = v_1 R^{11} R^4 y_1$,
  and so on. Therefore, in the Bockstein spectral sequence the
  elements $v_0^k R^4 y_1$ support differentials to the hidden
  extensions $v_0^k v_1 R^7 y_1 = v_1^{k+1} R^{7 + 2k} y_1$, the
  elements $v_0^k R^8 R^4 y_1$ support differentials to the hidden
  extensions $v_1^{k+1} R^{11+2k} R^4 y_1$, and so on. These impose
  the ``right-hand edge'' of the sawtooth patterns.
\end{proof}

\section{Further calculations: $n=2$}

The calculation for $n=2$ is carried out in a very similar fashion to
the calculation for $n=1$. Again, there are two spectral sequences
calculating $\Ext_{E(2)R\Mod}(QH_* BP, \mb F_2)$. The first is the
Bockstein spectral sequence
\[
  \Ext_{E(1)R\Mod}(QH_* BP, \mb F_2) \otimes \mb F_2[v_2] \Rightarrow
  \Ext_{E(2)R\Mod}(QH_* BP, \mb F_2),
\]
and the second is the spectral sequence
\[
  \Ext_{E(-1)R\Mod}(QH_* BP, \mb F_2) \otimes \mb F_2[v_0, v_1, v_2] \Rightarrow
  \Ext_{E(2)R\Mod}(QH_* BP, \mb F_2)
\]
arising from filtering the Koszul complex. The second has a basis of
admissible monomials
\[
  v_0^{j} v_1^{j} v_2^{l} R^{a_1} R^{a_2} \dots R^{a_m} y_1
\]
(except for those which are right multiples of $R^5$) with
$\mb F_2[v_0, v_1, v_2]$-linear differential satisfying the following
relations:
\begin{align*}
  dR^a(x) &= (a+1) (v_0 R^{a+1} x + v_1 R^{a+3} x + v_2 R^{a+7} x) + R^a(dx)\\
  R^a(v_2 x) &= 0\\
  R^a(v_1 x) &= v_2 R^{a+4} x\\
  R^a(v_0 x) &= v_2 R^{a+6} x + v_1 R^{a+2} x
\end{align*}
We can again use these differentials to determine hidden extensions
and Bockstein differentials on the basis calculated in the previous
section. We begin with some hidden extensions.
\begin{lem}
\label{lem:v2extension}
  In the $v_2$-Bockstein spectral sequence, when $x$ is a cycle there
  are hidden extensions
  \[
    v_0 R^{a_1} R^{a_2} x = v_2 R^{a_1 + 4} R^{a_2+2} x
  \]
  whenever $R^{a_1} R^{a_2} x$ is admissible and $a_1$ and $a_2$ are odd.
\end{lem}

\begin{proof}
  The differential
  \[
    d(R^{a_1-1} R^{a_2} x + R^{a_1} R^{a_2-1} x) = v_0 R^{a_1} R^{a_2}
    x + v_2 R^{a_1 + 4} R^{a_2+2} x,
  \]
  for $x$ a cycle in the filtered Koszul complex, establishes this
  relation.
\end{proof}

We now calculate differentials. The first type are differentials on
admissible monomials that start with two odd operations:
\[
  d(R^{a_1} R^{a_2} R^{a_3} x) =
  \begin{cases}
    0 &\text{if $a_1$, $a_2$, and $a_3$ are odd}\\
    v_2 R^{a_1 + 4} R^{a_2+2} R^{a_3+1} x &\text{if $a_1$, $a_2$ are odd and
    $a_3$ is even}
  \end{cases}
\]
As before, this makes all admissible monomials of the form $R^{a_1}
R^{a_2} R^{a_3} \dots R^{a_m} x$, where $a_1$, $a_2$, and $a_3$ are
odd, into $v_2$-torsion elements. These are systematically eliminated
by the admissible monomials of the form $R^{a_1} R^{a_2} R^{a_3} \dots
R^{a_m} x$, where $a_1$ and $a_2$ are odd but $a_3$ is even,
\emph{except} for those among the following list:
\begin{align*}
  &R^{a_1} R^{a_2} R^4 y_1,\\
  &R^{a_1} R^{a_2} R^8 R^4 y_1,\\
  &R^{a_1} R^{a_2} R^{16} R^8 R^4 y_1, \dots
\end{align*}

The second type are differentials on classes that are part of the
``sawtooth'' patterns. For $a$ odd, these differentials take the
following form:
\begin{align*}
  d(R^a y_1) &= 0,\\
  d(R^a R^4 y_1) &= v_2 R^{a+4} R^7 y_1,\\
  d(R^a R^8 R^4 y_1) &= v_2 R^{a+4} R^{11} R^4 y_1,\\
  d(R^a R^{16} R^8 R^4 y_1) &= v_2 R^{a+4} R^{19} R^8 R^4 y_1, \dots
\end{align*}
These differentials in the Bockstein spectral sequence hit
$v_0$-torsion elements. However, the hidden extension
\[
  v_0 R^{a_1} R^{a_2} x = v_2 R^{a_1 + 4} R^{a_2 + 2} x
\]
from Lemma~\ref{lem:v2extension} forces higher differentials:
\begin{align*}
  d(v_0^k R^a y_1) &= 0,\\
  d(v_0^k R^a R^4 y_1) &= v_2^{k+1} R^{a+4k + 4} R^{7+2k} y_1,\\
  d(v_0^k R^a R^8 R^4 y_1) &= v_2^{k+1} R^{a+4k + 4} R^{11 + 2k} R^4 y_1,\\
  d(v_0^k R^a R^{16} R^8 R^4 y_1) &= v_2 R^{a+4k + 4} R^{19 + 2k} R^8 R^4 y_1, \dots
\end{align*}
and so on. The targets span \emph{all} of the admissible monomials
which we previously described as being in the kernel of the differential.

Putting this calculation together, we are led to the following
conclusion.
\begin{prop}
  The $\Ext$-groups
  \[
    \Ext_{E(2)R\Mod}(QH_* BP, \mb F_2)
  \]
  are a direct sum of copies of four types of terms:
  \begin{enumerate}
  \item a free copy of $\mb F_2[v_0, v_1, v_2]$ generated by $y_1$,
  \item a $(v_0,v_1)$ sawtooth pattern generated by the classes $R^a
    y_1$ with $a$ odd, $a \geq 7$,
  \item $(v_0,v_1,v_2)$-torsion copies of $\mb F_2$ indexed by those admissible
    monomials $R^{a_1} R^{a_2} \dots R^{a_m} y_1$ with $a_1$, $a_2$,
    and $a_3$
    odd which do not end in any of the sequences $R^5 y_1$, $R^9 R^4 y_1$,
    $R^{17} R^8 R^4 y_1$, and so on, and
  \item $(v_0, v_2)$ sawtooth patterns
    \[
      \mb F_2[v_0,v_2] \{R^a R^b x\} / (v_0 R^a R^b x = v_2R^{a+4}
      R^{b+2} x)
    \]
    as $R^a R^b x$ range over admissible monomials with $a$ and $b$
    odd and $x$ in the set $y_1, R^4 y_1, R^8 R^4 y_1, \dots$, mod
    the relations that $R^5 y_1 = 0$, $R^9 R^4 y_1 = 0$, and so on.
  \end{enumerate}
\end{prop}

The first $(v_0,v_2)$ sawtooth pattern appears in
Figure~\ref{fig:pos2}.
\begin{figure}
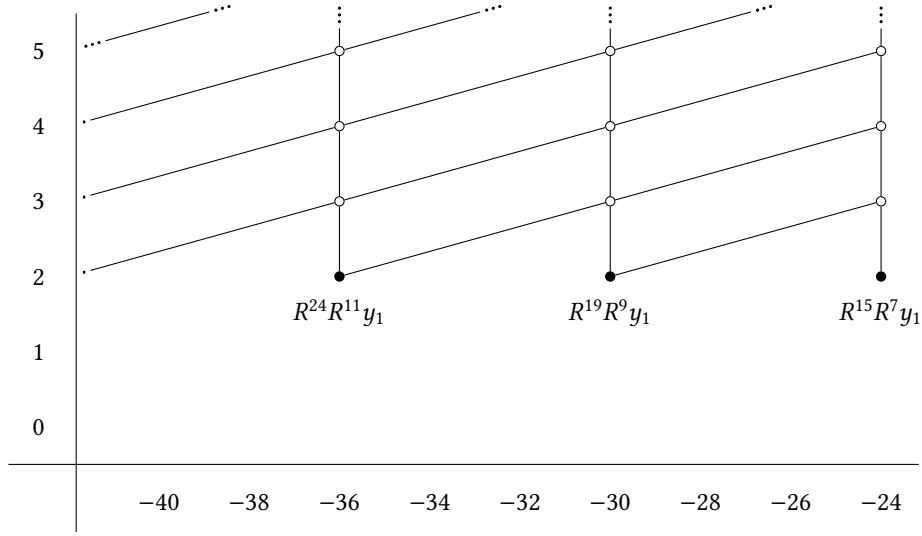

  \centering
\begin{sseqdata}[name=pos2, Adams grading, x tick step = 2, xscale =
    0.6, x range={-41}{-24}, y range={0}{5}]
    \class["R^{15} R^7 y_1" below, fill](-24,2)
    \class["R^{19} R^9 y_1" below, fill](-30,2) 
    \class["R^{24} R^{11} y_1" below, fill](-36,2)
    \class[fill](-42,2) 
    \foreach \x in {-24,-30,-36,-42} \foreach \y in {3, 4, 5, 6} {
      \class(\x,\y)
    } 
    \foreach \x in {-24,-30,-36} \foreach \y in {2, 3, 4, 5} {
      \structline(\x,\y)(\x,\y+1)
    } 
    \foreach \x in {-30,-36,-42} \foreach \y in {2, 3, 4, 5} {
      \structline(\x,\y)(\x+6,\y+1)
    }
  \end{sseqdata}
  \printpage[name = pos2, page = 1]
  
  \caption{Part of the first $(v_0,v_2)$ sawtooth pattern in $\Ext$ for $n=2$}
  \label{fig:pos2}
\end{figure}

\begin{rmk}
  The $v_1$-multiplication on the $(v_0, v_2)$-sawtooth patterns is
  more complicated and we will not describe it here.
\end{rmk}

\section{Final calculations in weight $2$}

The calculations we have been doing can be carried out for larger and
larger values of $n$ and obey systematic patterns, but require more
and more bookkeeping with respect to the edge cases. However, we will
ultimately be interested in weight $2$. Here, we have already isolated
everything: we have determined all the permanent cycles and their
quotient by the differentials on classes in lower weight.
\begin{prop}
  The part of
  \[
    Ext_{E(2)R\Mod}(QH_* BP, \mb F_2)
  \]
  on classes in weight less than or equal to $2$ is a sum of three
  terms:
  \begin{enumerate}
  \item in weight $0$, there is a copy of $\mb F_2[v_0, v_1, v_2]$
    generated by $y_1$;
  \item in weight $1$, there is a $(v_0, v_1)$-sawtooth pattern
    generated by the classes $R^a y_1$ for $a \geq 7$, $a$ odd, with
    $v_0 R^a y_1 = v_1 R^{a+2} y_1$;
  \item in weight $2$, there are $(v_0, v_2)$-sawtooth patterns
    generated by the classes $R^a R^b y_1$ for $b \geq 7$, $a > 2b$,
    $a$ and $b$ odd, with $v_0 R^a R^b y_1 = v_2 R^{a+4} R^{b+2}
    y_1$.\footnote{There is one sawtooth pattern for each odd positive
      value of $a-2b$.}
  \end{enumerate}
  In particular, these all lift to permanent cycles in the Koszul
  complex for $\Ext_{E(n)R\Mod}(QH_* BP, \mb F_2)$.
\end{prop}

In particular, the fact that we have permanent cycles implies that all
the higher Bockstein spectral sequences degenerate in weights less
than or equal to two.
\begin{cor}
  For any $n \geq 2$, the Bockstein spectral sequences
  \[
    Ext_{E(2)R\Mod}(QH_* BP, \mb F_2) \otimes \mb F_2[v_3,\dots,v_n]
    \Rightarrow Ext_{E(n)R\Mod}(QH_* BP, \mb F_2)
  \]
  degenerate in weights less than or equal to two.
\end{cor}

Since filtration is greater than or equal to weight, this gives a
complete list of classes in $\Ext$-filtration less than or equal to
three except for classes of the form $R^a R^b R^c x$, all of which are
in large negative degree.

\bibliography{../masterbib}
\end{document}